\documentclass[12pt]{article}

\usepackage{amssymb}
\usepackage{mathptmx}

\usepackage{amsfonts}
\usepackage{latexsym}
\usepackage{amsmath,amsthm}
\newtheorem{theorem}{Theorem}
\newtheorem{proposition}{Proposition}
\newtheorem{corollary}{Corollary}
\newtheorem{lemma}{Lemma}
\newtheorem{definition}{Definition}
\newtheorem{remark}{Remark}
\newtheorem{example}{Example}
{\bf}{\it}
\newcommand{\R}{\mathbb R}
\newcommand{\RR}{\overline{\mathbb R}}

\newcommand{\pr}{\prime}

\newcommand{\cl}[1]{{\rm cl}#1}

\title{Second-Order Karush--Kuhn--Tucker Optimality Conditions for Vector Problems with Continuously Differentiable Data and Second-Order Constraint Qualifications}
\author{Vsevolod I. Ivanov \thanks{Department of Mathematics, Technical University of Varna, Bulgaria.
E-mail: vsevolod.ivanov@tu-varna.bg}}

\begin{document}
\maketitle
\begin{abstract}
Some necessary and sufficient optimality conditions for inequality constrained problems with continuously differentiable data were obtained in the papers [I. Ginchev and V.I. Ivanov, Second-order optimality conditions for problems with C$\sp{1}$ data, J. Math. Anal. Appl., v. 340, 2008, pp. 646--657], [V.I. Ivanov, Optimality conditions for an isolated minimum of order two in C$\sp{1}$ constrained optimization, J. Math. Anal. Appl., v. 356, 2009, pp. 30--41] and [V. I. Ivanov, Second- and first-order optimality conditions in vector optimization, Internat. J. Inform. Technol. Decis. Making, 2014, DOI: 10.1142/S0219622014500540].

In the present paper, we continue these investigations. We obtain some necessary optimality conditions of Karush--Kuhn--Tucker type for scalar and vector problems. A new second-order constraint qualification of Zangwill type is introduced. It is applied  in the optimality conditions.


\end{abstract}

\section{Introduction}
\label{s12.3}
\setcounter{equation}{0}

What is the life of the man without his attempt to make the things in the best way? One of the tools for doing this is  mathematics and nonlinear programming in particular. Optimization became a self-dependent science after discovering the Kuhn--Tucker's optimality conditions.

Constraint qualifications (in short, CQ) play important role in the necessary optimality conditions. In this paper, we investigate second-order conditions and second-order CQ (in short, SOCQ). The SOCQ are usually connected to second-order local approximations of the feasible set. Historically, the first SOQC is due to McCormic \cite{mcc67}. In 1980 Ben-Tal obtained second-order 
Karush--Kuhn--Tucker (in short, KKT) conditions in terms of another SOCQ \cite{bt80}. Two years later Ben-Tal and Zowe \cite{ben82} derived second-order KKT conditions under another SOCQ. A SOQC of Guignard type was introduced and studied by Kawasaki \cite{kaw88}.
Aghezzaf and Hachimi \cite{agh99,hac07} obtained KKT conditions for multiobjective problems in terms of Abadie and Guignard types SOCQ. Generalizations of  Ban-Tal's SOCQ were applied by Penot \cite{pen98}, Jimenez and Novo \cite{jim04}. Optimality conditions were also obtained by Maciel, Santos, Sottosanto \cite{mac11} using two types SOCQ. 
All mentioned authors investigated twice differentiable problems or C$^2$ ones. SOCQ were applied in KKT conditions for C$^{1,1}$ problems in several works: Yang \cite{yan93} (extension of McCormic's SOCQ), 
Maeda \cite{mae04} 
(Abadie type SOCQ), Ginchev, Guerraggio, Rocca \cite{gin05} (Kuhn-Tucker type SOCQ). Second-order necessary conditions of KKT type for problems wirh continuously differentiable data were derived by Ivanov \cite{IJITDM} with the help of a SOCQ of Mangasarian-Fromovitz type.  Several more papers derived KKT conditions with first-order CQ. For example, Andreani, Echag\"ue and Schuverdt \cite{and10} applied recently in second-order results the first-order  CQ, which is called the constant rank condition.  

In the papers \cite{jmaa-1,jmaa-3}, the authors obtained various second-order optimality conditions for the nonlinear programming problem with inequality constraints

\bigskip\noindent
Minimize $\quad f(x)\quad$subject to$\quad 
g_i(x)\leqq0,\;i=1,2,...,m,$\hfill ${\rm (P)}$
\bigskip

\noindent
where the real functions $f$, $g_i$, $i=1,2,...,m$ are defined on some open set $X$ and $X\subset\R^s$.
All results are derived for functions, which do not satisfy the standard assumptions for second-order Fr\'echet differentiability. In the most results, the objective function and the constraint are continuously differentiable and the standard second-order directional derivative is applied. Some more optimality conditions for the vector problem with continuously differentiable data were obtained also by Ivanov \cite{IJITDM}.

In the present work, we continue the investigations given there. We consider the vector problem

\bigskip\noindent
Minimize $\quad f(x)\quad$subject to$\quad 
g(x)\leqq0,$\hfill ${\rm (VP)}$
\bigskip

\noindent
where $f:X\to\R^n$ and $g:X\to\R^m$ are given vector functions defined on some open set $X$ and $X\subset\R^s$.
 We introduce a new second-order CQ, which is analogous to the Zangwill CQ \cite{zan69,gio04}. It is more general than the SOCQ, introduced in \cite{jmaa-3}. We obtain second-order  KKT necessary optimality conditions for a weak local minimum in the problem (P) in terms of this CQ. Our SOCQ fits to problems with C$^1$ data. In our knowledge, it is an open question to apply SOCQ in such problems. Theorem \ref{dualcond} and Corollary \ref{dualcond(P)} generalize the first-order KKT necessary conditions in terms of Abadie CQ \cite{aba67} and Guignard CQ \cite{gui69}.  It is an open question to obtain second-order conditions that are generalizations of all first-order KKT ones. 
In the cited works the authors did not obtain such results, because they do not consider problems with arbitrary differentiable data like a lot of the first-order known results. They consider problems with twice differentiable or at least C$^{1,1}$ data. The second-order linearizing cone that we define is different from the second-order lnearizing cone from the paper of Kawasaki \cite{kaw88}. It is also different from the one, which is defined by Aghezzaf and Hachimi \cite{agh99}, but in principal both cones are similar.  


\bigskip
We proceed this section with recalling the definitions of some preliminary notions and notations.
Denote by $\R$ the set of reals and let $\RR=\R\cup\{-\infty\}\cup\{+\infty\}$, by ${\rm cl}(S)$ the closed hull of the set $S$, and by ${\rm conv}(S)$ the convex hull of $S$.

Consider the problem {\rm{(VP)}}. Denote by $S$ the set of feasible points, that is
\[
S:=\{x\in X\mid g_i(x)\le 0,\; i=1,2,...,m\}.
\]
For every feasible point $x\in S$, let $I(x)$ be the set of active constraints
\[
I(x):=\{i\in\{1,2,...,m\}\mid g_i(x)=0\}.
\]

Throughout this paper, we use the following notations comparing the vectors $x$ and $y$ with components $x_i$ and $y_i$ in finite-dimensional spaces:
\[
\begin{array}{c}
x<y\quad\textrm{if}\quad x_i<y_i\quad\textrm{for all indices }i; \\
x\leqq y\quad\textrm{if}\quad x_i\leqq y_i\quad\textrm{for all indices }i; \\
x\le y\quad\textrm{if}\quad x_i\leqq y_i\quad\textrm{for all indices }\quad i\quad\textrm{with at least one being strict.}
\end{array}
\]

\begin{definition}
A feasible point $\bar x\in S$ is called a weak local Pareto minimizer, or weakly efficient iff there exists a neigbourhood $U\ni\bar x$ such that there is no $x\in U\cap S$ with  $f(x)<f(\bar x)$. 
\end{definition}

\begin{definition}
A direction $d$ is called critical at the point $x\in S$ iff
\[\nabla f_j(x)d\leqq 0\;\text{ for all }\; j\in \{1,2,\dots,n\}\quad\text{and}\quad \nabla g_i(x)d\leqq 0\;\text{ for all }\; i\in I(x).\]
\end{definition}

For a feasible point $\bar x$ and a direction $d$, denote by $J(\bar x,d)$ and  $K(\bar x,d)$ the following sets;
\[
J(\bar x,d):=\{j\in\{1,2,\dots,n\}  \mid\nabla f_j(\bar x)d=0\},\quad
K(\bar x,d):=\{i\in  I(\bar x)\mid\nabla g_i(\bar x)d=0\}.
\]
These notations are sensible only if $\bar x$ is a local minimizer and $d$ is a critical direction.

\begin{definition}
Let the function $h:X\to\R$ with an open domain $X\subset\R^s$ be Fr\'echet differentiable at
the point $x\in X$. Then the second-order directional derivative
$h^{\pr\pr}(x,u)$ of $h$ at the point $x\in X$ in direction $u\in\R^n$
is defined as an element of $\R$ by the equality
\begin{displaymath}
h^{\pr\pr}(x,u)=\lim_{t\to +0}\,2t^{-2}[h(x+tu)-h(x)-t\nabla h(x)u].
\end{displaymath}
The function $h$ is called second-order directionally differentiable on $X$
iff the derivative $h^{\pr\pr}(x,u)$ exists for each $x\in X$ and any
direction $u\in\R^n$ and it is finite.
\end{definition}
\begin{definition}[\cite{gio04}]
Let the function $h: X\to\R$ with an open domain $X\subset\R^s$ be Fr\'echet 
differentiable at the point $x\in X$. Then $h$ is said to be pseudoconvex
at $x\in X$ iff 
\[
y\in X,\; h(y)<h(x)\quad\textrm{imply}\quad \nabla h(x)(y-x)<0.
\]
If $h$ is differentiable on $X$, then it is called
pseudoconvex on $X$ when $h$ is pseudoconvex at each $x\in X$. If the function $-h$ is pseudoconvex, then $h$ is called pseudoconcave.
\end{definition}

The following definition is due to Ginchev and Ivanov \cite{gi05}.
\begin{definition}
Consider a function $h:X\to\R$ with an open domain $X$, which is Fr\'echet  differentiable at $x\in X$ and
second-order directionally differentiable at $x\in X$ in every direction
$y-x$ such that $y\in X$, $h(y)<h(x)$, $\nabla h(x)(y-x)=0$.
The function $h$ is called second-order pseudoconvex at $x\in X$ iff  for all $y\in X$ the following implications hold:
\[
h(y)<h(x)\quad\mbox{implies}\quad\nabla h(x)(y-x)\leqq0;
\]
\[
h(y)<h(x),\;\nabla h(x)(y-x)=0\quad\mbox{imply}\quad h^{\pr\pr}(x,y-x)<0.
\]
Suppose that  $h$ is differentiable on $X$ and
second-order directionally differentiable at every $x\in X$ in each direction
$y-x$ such that $y\in X$, $h(y)<h(x)$, $\nabla h(x)(y-x)=0$.
The function $h$ is called second-order pseudoconvex on  $X$ iff it is
second-order pseudoconvex at every $x\in X$.  If $-h$ is second-order pseudoconvex, then $h$ is called second-order pseudoconcave.
\end{definition}
It follows from this definition that every differentiable pseudoconvex function is second-order pseudoconvex. The converse does not hold.

\section{A new second-order constraint qualification of Zangwill type}
\label{s11new}

Consider the problem (VP) and the following conditions:
\[
\left.
\begin{array}{@{}l}
\textrm{The functions } g_i,\; i\notin I(\bar x) \textrm{ are continuous at } \bar x; \\

\textrm{the functions } f_j,\; g_i,\; j=1,2,\dots,n,\; i\in I(\bar x) \textrm{ are continuosly diffrentiable;} \\
\textrm{If } \nabla f_j(\bar x)d=0, \textrm{ then there exists  }f^{\pr\pr}_j(\bar x,d),\\
\textrm{If } \nabla g_i(\bar x)d=0,\; i\in I(\bar x), \textrm{ then there exists  }g_i^{\pr\pr}(\bar x,d),\\
\end{array}\right\}\hfill \eqno{\rm (C)}
\]

For every feasible point $x$ and direction $d$, consider the sets:
\begin{gather*}
A(x,d):=\{z\in\R^n\mid\forall\; i\in K(x,d)\;\exists\delta_i>0:
g_i(x+td+0.5t^2z)\le 0,\quad\forall\; t\in(0,\delta_i)\}, \\
B(x,d):=\{z\in\R^n\mid\nabla g_i(x)z+g_i^{\pr\pr}(x,d)\le 0
\textrm{ for each } i\in K(x,d)\}.
\end{gather*}
By definition $A(x,d)=B(x,d)=\R^n$ if $K(x,d)=\emptyset$.

\begin{proposition}\label{lema1}
Let $\bar x$ be a feasible point for the Problem {\rm (VP)} and $d$ be a direction. Suppose that all functions $g_i$, $i\in I(\bar x)$ are continuously differentiable and there exist $g_i^{\pr\pr}(\bar x,d)$, $i\in I(\bar x)$, provided that $\nabla g_i(\bar x)d=0$. Then $ A(\bar x,d)\subset B(\bar x,d)$.
\end{proposition}
\begin{proof}
Suppose that $i\in I(\bar x)$ with $\nabla g_i(\bar x)d=0$ and $z\in A(\bar x,d)$.
Then there exists $\delta_i>0$ such that 
\begin{equation}\label{2}
g_i(\bar x+td+0.5t^2z)-g_i(\bar x)\le 0,\quad\forall\; t\in(0,\delta_i).
\end{equation}
Consider the function of one variable
$\varphi_i(t)=g_i(\bar x+td+0.5t^2z)$.
Since $X$ is open and $\bar x$ is feasible, then there exists a number $\delta_i>0$ such that $\varphi_i$ is defined for all numbers $t$ with $-\delta_i<t<\delta_i$. The following equality holds:
\[
\varphi^\pr_i(t)=\nabla g_i(\bar x+td+0.5t^2z)(d+tz).
\]
Therefore $\varphi^\pr_i(0)=\nabla g_i(\bar x)d$. Consider the differential quotient
\[
2t^{-2}[\varphi_i(t)-\varphi_i(0)-t\varphi^\pr_i(0)]
=2t^{-2}[g_i(\bar x+td+0.5t^2 z)-g_i(\bar x)-t\nabla g_i(\bar x)d].
\]
Let us choose an arbitrary sequence $\{t_k\}_{k=1}^\infty$ of positive numbers converging to 0.
According to the mean-value theorem, for every positive integer $k$ there exists $\theta^i_k\in(0,1)$ with
\begin{equation}\label{1}
g_i(\bar x+t_k d+0.5 t_k^2 z)=g_i(\bar x+t_k d)+\nabla g_i(\bar x+t_k d+0.5 t_k^2 \theta^i_k z)(0.5 t_k^2 z).
\end{equation}
It follows from $g_i\in\rm{C}^1$  
 and (\ref{1}) that
\[
\varphi^{\pr\pr}_i(0,1)=\lim_{k\to+\infty}[\nabla g_i(\bar x+t_k d+0.5t^2_k\theta^i_k z)z+
\]
\[
2t^{-2}_k (g_i(\bar x+t_k d)-g_i(\bar x)-t_k\nabla g_i(\bar x)d)]=\nabla g_i(\bar x)z+g^{\pr\pr}_i(\bar x,d).
\]
Therefore 
\begin{equation}\label{3}
\nabla g_i(\bar x)z+g^{\pr\pr}_i(\bar x,d)=\varphi^{\pr\pr}_i(0,1).
\end{equation}
It follows from (\ref{2}) and (\ref{3}) that
\[
\nabla g_i(\bar x)z+g_i^{\pr\pr}(\bar x,d)=\lim_{t\to +0}\, 2t^{-2}[g_i(\bar x+td+0.5t^2 z)-g_i(\bar x)]\le0,
\]
which proves that $A(\bar x,d)\subset B(\bar x,d)$. 
\end{proof}

The following example shows that the converse claim of Proposition \ref{lema1} does not hold.
\begin{example}
Consider the function $g:\R^2\to\R$, defined by $g(x_1,x_2)=x_1^3$. Choose
$\bar x=(0,0)$, $d=(1,0)$. We have
\[
A(\bar x,d)=\{z\in\R^2\mid\nabla g(\bar x)(d)=0\;\Rightarrow\; \exists\delta>0:
g(\bar x+td+0.5t^2z)\le 0\;\forall\; t\in(0,\delta)\}
\]
\[
B(\bar x,d)=\{z\in\R^2\mid \nabla g(\bar x)(d)=0\textrm{ implies }
\nabla g(\bar x)z+g^{\pr\pr}(\bar x,d)\le 0\},
\]
$\nabla g(\bar x)=(0,0)$, $g^{\pr\pr}(\bar x,d)=0$. If  $z=(1,0)$, then
$g(\bar x+td+0.5t^2 z)>0$ for all $t>0$,
Therefore $z\in B(\bar x,d)$, but $z\notin A(\bar x,d)$.
\end{example}


\begin{definition}
Consider a function of one variable $\varphi:(-a,a)\to\R$, which is Fr\'echet differentiable at the point $t=0$ and there exists its second-order right derivative

\centerline{$\varphi^{\pr\pr}(0,1):=\lim_{t\to +0}2 t^{-2}\left[\varphi(t)-\varphi(0)-t\varphi^\pr(0)\right].$}

\noindent
Then we call $\varphi$ second-order locally pseudoconcave at $t=0$ on the right, iff there exists $\delta>0$ such that
\[
\begin{array}{c}
\varphi(t)>\varphi(0),\; 0<t<\delta\quad\textrm{implies}\quad\varphi^\pr(0)\ge 0,\\
\varphi(t)>\varphi(0),\; 0<t<\delta,\; \varphi^\pr(0)=0\quad\textrm{implies}\quad\varphi^{\pr\pr}(0,1)>0.
\end{array}
\]
\end{definition}

The condition the constraint functions $g_i$, $i\in I(\bar x)$ to be pseudoconcave at $\bar x$, where $\bar x$ is the local minimizer, is called the weak reverse constraint qualification \cite[p. 253]{gio04}. The respective second-order condition is the assumption that $g_i$, $i\in I(\bar x)$ are second-order pseudoconcave at $\bar x$. This condition is weaker than the respective first-order one, because every pseudoconvex function is second-order pseudoconvex, but the inverse claim is not true. The CQ that the functions of one variable $\varphi_i(t)$, which are defined by the equality 
\begin{equation}\label{12.65}
\varphi_i(t)=g_i(\bar x+td+0.5t^2z),\quad t\in\R
\end{equation}
are second-order locally pseudoconcave at $t=0$ on the right, is a weaker second-order CQ.

\begin{proposition}
Let the constraint functions satisfy Conditions {\rm (C)}.
Suppose that $\bar x$ is a feasible point for {\rm (VP)} and $d$ is an arbitrary direction. Let the functions of one variable $\varphi_i$, $i\in K(\bar x,d)$, defined by {\rm (\ref{12.65})}, be second-order locally pseudoconcave at the point $t=0$ on the right  for every $z\in\R^n$. Then

\centerline{$A(\bar x,d)=B(\bar x,d).$}
\end{proposition}
\begin{proof} 
According to Proposition \ref{lema1} it is enough to prove that $B(\bar x,d)\subseteq A(\bar x,d)$. If $K(\bar x,d)=\emptyset$, then the claim is obvious. Suppose the contrary that there exists $z\in B(\bar x,d)$, but $z\notin A(\bar x,d)$. 
It follows from $z\notin A(\bar x,d)$ that there exists
$j\in K(\bar x,d)$ and a sequence $\{t_k\}_{k=1}^\infty$, $t_k\to +0$, which consists of positive numbers with the property $\varphi_j(t_k)>\varphi_j(0)$ for each positive integer $k$. By second-order local pseudoconcavity we obtain that 
 $\varphi_j^{\pr\pr}(0,1)>0$, which implies that $z\notin B(\bar x,d)$, a contradiction.
\end{proof}

If $\bar x$ is a feasible point, then the set $B(\bar x,d)$ is closed, but $A(\bar x,d)$ is not. 

\begin{example}
Let $S=\{x=(x_1,x_2)\in\R^2\mid x_1^2-x_2\le 0\}$, $\bar x=(0,0)$, $d=(0,0)$. Then $(1,z_2)\in A(\bar x,d)$ with $z_2$ arbitrary positive number, but $(1,0)\notin A(\bar x,d)$. Therefore, $A(\bar x,d)$ is not closed.
\end{example}


\begin{definition}
We introduce the condition $\cl(A(\bar x,d))=B(\bar x,d)$ under the condition that the nonactive constraints are continuous. In the next section, we show that it is a SOCQ.
\end{definition}
We could name this SOCQ the second-order Zangwill CQ, because it is a second-order analog of the Zangwill CQ \cite{zan69} $\cl(Z(\bar x))=L(\bar x)$, where
\[
L(\bar x)=\{d\in\R^n\mid\nabla g_i(\bar x)d\leqq 0,\; i\in I(\bar x)\}
\] 
is the linearizing cone of the problem (VP) at the feasible point $\bar x$ and
\[
Z(\bar x)=\{d\in\R^n\mid\exists\delta>0: \bar x+td\in S,\quad\forall\; t\in(0,\delta)\}
\]
is the cone of the feasible directions to $S$ at $\bar x$ under the assuption that the nonactive constraints are continuous at the feasible point $\bar x$. In

\section{Second-order KKT necessary conditions for weak local minimum}

\begin{theorem}[Primal conditions]\label{OptCond}
Let $\bar x$ be a weak local minimizer of the problem {\rm (VP)} and $d$ be a critical direction.
Suppose that Conditions {\rm (C)} are satisfied. 
Assume that the constraint qualification $\cl(A(\bar x,d))=B(\bar x,d)$ holds. Then there does not exist a vector $z$ such that

\begin{gather}
\nabla f_j(\bar x)z+ f^{\pr\pr}_j(\bar x,d)<0,\quad j\in J(\bar x,d), \label{11.80}\\
\nabla g_i(\bar x)z+g_i^{\pr\pr}(\bar x,d)\leqq 0,\quad i\in K(\bar x,d)\label{11.81}
\end{gather} 
\end{theorem}
\begin{proof}
The conditions (\ref{11.80}), (\ref{11.81}) can be considered as a system of inequalities. This system contains at least one inequality, because $\bar x$ is a weak local solution.
Assume the contrary that there exists a vector $z$, which satisfies (\ref{11.80}) и (\ref{11.81}). Therefore $z\in B(\bar x,d)$. Consider the following cases:

1) Let  $i\in K(\bar x,d)$. It follows from the condition $\cl(A(\bar x,d))=B(\bar x,d)$ that there exists a sequence $\{z_l\}_{l=1}^\infty$, converging to $z$, such that $z_l\in A(\bar x,d)$. Take an arbitrary positive integer $l$. Suppose that it is fixed. Therefore, there exists a number  $\delta_i>0$ with $g_i(\bar x+td+0.5t^2z_l)\le 0$ for every $t\in (0,\delta_i)$. 

2) Suppose that $i\in I(\bar x)\setminus K(\bar x,d)$.
We have $\nabla g_i(\bar x)d<0$. Therefore $\varphi^\pr_i(0)<0$, where $\varphi_i(t)=g_i(\bar x+td+0.5t^2z_l)$. It follows from here that there exists $\delta_i>0$ with $\varphi_i(t)<\varphi_i(0)$, that is $g_i(\bar x+td+0.5t^2z_l)<g_i(\bar x)=0$ for all $t\in (0,\delta_i)$. 

3) For every $i\in\{1,2,...,m\}\setminus I(\bar x)$ is satisfied the inequality $g_i(\bar x)<0$. According to the assumption that $g_i$ is continuous, there exists $\delta_i>0$ such that $g_i(\bar x+td+0.5t^2z_l)< 0$  for all $t\in(0,\delta_i)$. 

Thus, we obtain from all these cases that the point $\bar x+td+0.5t^2z_l$ is feasible for all sufficiently small positive numbers $t$. 

We consider two cases concerning the objective function.

1) Let $\nabla f_j(\bar x)d<0$. Define the function of one variable
$\psi_j(t)=f_j(\bar x+td+0.5t^2z_l)$.
Then $\psi^\pr_j(0)<0$ and hence,  there exists $\varepsilon_j>0$ with $f_j(\bar x+t d+0.5t^2 z_l)<f_j(\bar x)$ for arbitrary $t\in(0,\varepsilon_j)$.  

2) Let $j\in J(\bar x,d)$ that is $\nabla f_j(\bar x)d=0$. 
Since $X$ is open and $\bar x$ is feasible, then there exists a number $\varepsilon_j>0$ such that $\psi_j$ is defined for all numbers $t$ with $-\varepsilon_j<t<\varepsilon_j$. The following equality holds:
\[
\psi^\pr_j(t)=\nabla f_j(\bar x+td+0.5t^2z_l)(d+tz_l).
\]
Therefore $\psi^\pr_j(0)=\nabla f_j(\bar x)d$. Consider the differential quotient
\[
2t^{-2}[\psi_j(t)-\psi_j(0)-t\psi^\pr_j(0)]
=2t^{-2}[f_j(\bar x+td+0.5t^2 z_l)-f_j(\bar x)-t\nabla f_j(\bar x)d].
\]
Let us choose an arbitrary sequence $\{t_k\}_{k=1}^\infty$ of positive numbers, converging to 0.
According to the mean-value theorem, for every positive integer $k$, there exists $\theta^k_j\in(0,1)$ with
\begin{equation}\label{12.63}
f_j(\bar x+t_k d+0.5 t_k^2 z_l)=f_j(\bar x+t_k d)+\nabla f_j(\bar x+t_k d+0.5 t_k^2 \theta^k_j z_l)(0.5 t_k^2 z_l).
\end{equation}
It follows from $f\in\rm{C}^1$  
 and (\ref{12.63}) that
\[
\psi^{\pr\pr}_j(0,1)=\lim_{k\to+\infty}[\nabla f_j(\bar x+t_k d+0.5t^2_k\theta^k_j z_l)z_l+
\]
\[
2t^{-2}_k (f_j(\bar x+t_k d)-f_j(\bar x)-t_k\nabla f_j(\bar x)d)]=\nabla f_j(\bar x)z_l+f^{\pr\pr}_j(\bar x,d).
\]
Therefore, $\nabla f_j(\bar x)z_l+f^{\pr\pr}_j(\bar x,d)=\psi^{\pr\pr}_j(0,1)$. 
It follows from (\ref{11.80}) that $\nabla f_j(\bar x)z_l+ f^{\pr\pr}_j(\bar x,d)<0$ for all suficiently large numbers $l$. Therefore $\psi^{\pr\pr}_j(0,1)<0$. By $j\in J(\bar x,d)$ we have 
\[
\lim_{t\to+0}\, 2[\psi_j(t)-\psi_j(0)]/t^2<0,
\]
 which implies that there exists $\varepsilon_j>0$ such that $f_j(\bar x+t d+0.5t^2 z_l)<f_j(\bar x)$ for arbitrary $t\in(0,\varepsilon_j)$. 

Taking into account both cases, we get a contradiction to the hypothesis that $\bar x$ is a weak local minimizer, since the inequality  $f_j(\bar x+t d+0.5t^2 z_l)<f_j(\bar x)$ is satisfied for all $t\in (0,\varepsilon)$, where
$\varepsilon$ is the minimal among the positive numbers $\varepsilon_j$ and $\delta_i$.
\end{proof}

Let us consider the system with unknowns $u\in\R^n$ and $v\in\R$, where $d$ is an arbitrary critical direction:
\begin{equation}\label{11.63}
\left\{\begin{array}{l}
\nabla f_j(\bar x)u + v f^{\pr\pr}_j(\bar x,d)<0,\quad j=1,2,\dots,n \\
\nabla g_i(\bar x)u + v g_i^{\pr\pr}(\bar x,d)\leqq 0,\quad i\in I(\bar x),  \\
v>0
\end{array}\right.
\end{equation}
and the system with an unknown $u\in\R^n$:
\begin{equation}\label{11.64}
\begin{array}{l}
\nabla f_j(\bar x)u<0, \;j=1,2,\dots,n,\quad\nabla g_i(\bar x)u\leqq 0,\quad i\in I(\bar x), 
\end{array}
\end{equation}
where $\bar x$ and $d$ are a given point and a direction respectively.

\begin{lemma}\label{KuhnTucker}
Let the point $\bar x$ be a weak local solution of the problem {\rm (VP)} and $d$ be a nonzero critical direction.
Suppose that the functions $f$, $g_i$, $i\in I(\bar x)$ are Fr\'echet differentiable, the functions $g_i$, $i\notin I(\bar x)$ are continuous and in the case when $\nabla f_j(\bar x)d=0$ or $\nabla g_i(\bar x)d=0$, $i\in I(\bar x)$, there exist the second-order directional derivatives $f^{\pr\pr}_j(\bar x,d)$ or $g_i^{\pr\pr}(\bar x,d)$ respectively. Then a necessary and sufficient condition for the existance of Lagrange multipliers $\lambda=(\lambda_1,\dots,\lambda_n)$ and $\mu=(\mu_1,\dots,\mu_m)$, $\lambda\ge 0$, $\mu\geqq 0$, which satisfy KKT conditions 
\begin{equation}\label{11.65}
\begin{array}{l}
\mu_i g_i(\bar x)=0,\;i=1,2,...,m,\quad\nabla L(\bar x)=0 \\
 L^{\pr\pr}(\bar x,d)=\sum_{j=1}^n \lambda_j f^{\pr\pr}_j(\bar x,d)+\sum_{i\in I(\bar x)}\mu_i g_i^{\pr\pr}(\bar x,d)\geqq 0,
\end{array}\end{equation}
where $L$ is the Lagrange function $L=\sum_{j=1}^n \lambda_j f_j+\sum_{i=1}^m\,\mu_i g_i$,
is the condition that both systems (\ref{11.63}) and (\ref{11.64}) are not solvable.
\end{lemma}
\begin{proof}
Let $d$ be an arbitrary critical direction. Consider the linear programming problem

\bigskip
\begin{tabular}{ll}
Maximize & $0$         \\
subject to & $\nabla f_j(\bar x)u+v f^{\pr\pr}_j(\bar x,d)\le -1$,\; $j=1,2,\dots,n$ \\
& $\nabla g_i(\bar x)u+v g^{\pr\pr}_i(\bar x,d)\le 0,\; i\in I(\bar x)$ \\
& $v\geqq 0$.
\end{tabular}
\bigskip

Its dual is the following problem:

\bigskip
\begin{tabular}{ll}
Minimize & $-\sum_{j=1}^n\lambda_j$ \\
subject to & $\sum_{j=1}^n\lambda_j\nabla f_j(\bar x)+\sum_{i\in I(\bar x)}\mu_i\nabla g_i(\bar x)=0$, \\
& $\sum_{j=1}^n\lambda_j f^{\pr\pr}_j(\bar x,d)+\sum_{i\in I(\bar x)}\mu_i g^{\pr\pr}_i(\bar x,d)\ge 0,$ \\
& $\lambda\geqq 0$,\quad $\mu_i\geqq 0,\quad\forall i\in I(\bar x).$
\end{tabular}
\bigskip

\noindent
Suppose that the systems (\ref{11.63}) and (\ref{11.64}) have no solutions. Therefore, the primal problem is not solvable, because it is infeasible. According to duality theorem the dual problem also is not solvable. Since $\lambda=0$, $\mu_i=0$, $i\in I(x)$ is a feasible point, then the dual problem is unbounded from below. Therefore, there exist Lagrange multipliers, which satisfy the second-order Karush--Kuhn--Tucker conditions.

Conversely, let there exist Lagrange multipliers, which satisfy KKT conditions. Therefore, the dual problem has no solutions, because its objective function is unbounded from below over the feasible set. It follows from duality theorem that the primal problem is unsolvable, because it is infeasible. Therefore, there are no a vector $u\in\R^n$ and a number $v>0$, which form a solution of the system (\ref{11.63}), there is no a vector $u\in\R^n$, which forms a feasible point for the primal problem together with the number $v=0$. We obtain from here that the system (\ref{11.64}) is inconsistent.
\end{proof}





Let $S$ be a given set. The Bouligand tangent cone (or the contingent cone) \cite{gio04} of the set $S$
at the point $x\in\mathrm{cl}(S)$ is defined as follows:
\[
\begin{array}{c}
T(S,x):=\{u\in\R^n\mid\exists \{t_k\}, t_k>0, t_k\to +0,
\exists \{u_k\}\subset\R^n, \\
u_k\to u\textrm{ such that }
x+t_ku_k\in S\textrm{ for all positive integers }k\}.
\end{array}
\]

If $S$ is the feasible set of the problem (VP), then the condition $L(\bar x)=T(S,\bar x)$ is called the Abadie CQ \cite{aba67}.

\begin{theorem}[Dual conditions]\label{dualcond}
Let $\bar x$ be a weak local minimizer for the problem {\rm (VP)} and let $d$ be a nonzero critical direction. Suppose that Conditions {\rm (C)} are satisfied.
Suppose that the constraint qualification 
$\cl(A(\bar x,d))=B(\bar x,d)$ and the Abadie CQ  hold.
Then there exist nonnegative Lagrange multipliers $\lambda=(\lambda_1,\dots,\lambda_n)$, $\lambda\ne 0$ and $\mu=(\mu_1,...,\mu_m)$, which satisfy the second-order KKT conditions (\ref{11.65}).
\end{theorem}
\begin{proof}
Let $d$ be an arbitrary critical direction. We prove that the system (\ref{11.63}) has no solutions. Let us suppose the contrary that the system (\ref{11.63}) is solvable and let $(u,v)$ be an arbitrary solution. It follows from here that there exists a point $z$, which satisfies conditions (\ref{11.80}) and (\ref{11.81}). This is a contradiction to Theorem \ref{OptCond}. 

We prove that the system (\ref{11.64}) has no solutions. Assume the contrary and let $u\in\R^n$ be a solution. Therefore, by the definition of the linearing cone, $u\in L(\bar x)$. It follows from Abadie CQ that $u\in T(S,\bar x)$.
Let $F$ be the cone $F=\{d\mid\nabla f_j(\bar x)d<0,\; j=1,2,\dots,n\}$. It is known \cite[Theorem 6.6.1]{gio04} that $F\cap T(S,\bar x)=\emptyset$. On the other hand, by Abadie CQ, we have $u\in F\cap T(S,\bar x)$, which is a contradiction.

It follows from our arguments up to here that both systems (\ref{11.63}) and (\ref{11.64}) are not consistent.
Then according to Lemma \ref{KuhnTucker}  there exist Lagrange multipliers, which satisfy the second-order KKT conditions.
\end{proof}

The closed convex hull of the Bouligand tangent cone is called the pseudotangent cone \cite{gui69}, that is 
\[
PT(S,x):=\mathrm{cl}\,(\mathrm{conv}\, T(S,x)).
\]

If $S$ is the feasible set of the problem (P), then the condition $L(\bar x)=PT(S,\bar x)$ is called the Guignard CQ \cite{gui69}.

\begin{corollary}[Dual conditions for the scalar problem (P)]\label{dualcond(P)}
Let $\bar x$ be a local minimizer for the scalar problem {\rm (P)} and let $d$ be a nonzero critical direction. Suppose that Conditions {\rm (C)} are satisfied.
Suppose that the constraint qualification 
$\cl(A(\bar x,d))=B(\bar x,d)$ and the Guignard CQ  hold.
Then there exist nonnegative Lagrange multipliers $\mu_1,...,\mu_m$, which satisfy the second-order KKT conditions (\ref{11.65}) with $n=1$.
\end{corollary}
\begin{proof}
We should prove only the part that the system (\ref{11.64}) has no solutions. Assume the contrary and let $u\in\R^n$ be a solution. Therefore, by the definition of the linearing cone, $u\in L(\bar x)$. It follows from Guignard CQ that 
$u\in PT(S,\bar x)$. On the other hand, it is well-known (see Lemma 5.1.2 from the book \cite{baz79}) that $\nabla f(\bar x)d\ge 0$ for all $d\in T(S,\bar x)$, where $T(S,\bar x)$ is the Bouligand tangent cone to the feasible set $S$ at $\bar x$. It follows from here that  $\nabla f(\bar x)d\ge 0$ for all $d\in PT(S,\bar x)$. In particular,  $\nabla f(\bar x)u\ge 0$, which contradicts the assuption that $u$ is a solution of the system (\ref{11.64}).
\end{proof}

\begin{remark}
In the case when $d=0$, Corollart \ref{dualcond(P)} reduce to the first-order KKT optimality conditions with Guignard CQ. Indeed, this direction is critical, the second-order derivatives exist and they are equal to zero. The set $A(\bar x,d)$ coincides with the cone of the feasible directions, the set $B(\bar x,d)$ coincides with the linearizing cone, the second-order Zangwill CQ reduce to Zangwill CQ, the conditions (\ref{11.65}) reduce to KKT necessary optimality conditions. Threfore, these necessary conditions are particular case of Corollary \ref{dualcond(P)}.
\end{remark}

\end{document}